\documentclass{amsart}
\pdfoutput=1
\usepackage{mathrsfs}
\usepackage{amsmath}
\usepackage{amsthm}
\usepackage{amssymb}
\usepackage{indentfirst}

\usepackage{geometry}             
\usepackage[dvipsnames]{xcolor}                                         
\usepackage{graphicx}                           
\usepackage{amssymb,mathrsfs}
\usepackage{amsthm,amsmath}
\usepackage{stmaryrd}
\usepackage{tikz}
\usepackage{tikz-cd}
\usepackage{accents,upgreek,enumerate}
\usepackage[headings]{fullpage}
\usepackage{bm}
\usepackage{mathtools}
\usepackage[all]{xy}
\usepackage{caption}
\usepackage[euler-digits]{eulervm}
\usepackage[normalem]{ulem}

\usepackage[utf8x]{inputenc}

\usepackage{ifthen}

\theoremstyle{plain} 
\newtheorem{prop}{Proposition}[section]
\newtheorem{thm}[prop]{Theorem}
\newtheorem{lem}[prop]{Lemma} 
\newtheorem{cor}[prop]{Corollary}

\newtheorem{rmk}[prop]{Remark}

\theoremstyle{remark}

\theoremstyle{definition}
\newtheorem{defn}[prop]{Definition}
\newcommand{\hk}{hyperk\"{a}hler }

\newcommand{\kntype}{$K3^{[n]}$ type}

\newcommand{\Sym}{\text{Sym}}

\newcommand{\R}{\mathbb R}
\newcommand{\Q}{\mathbb Q}
\newcommand{\C}{\mathbb C}
\newcommand{\Z}{\mathbb Z}

\newcommand{\rank}{\mathrm{rank}}
\newcommand{\disc}{\mathrm{disc}}
\newcommand{\Kum}{\mathrm{Kum}}

\begin{document}

\makeatletter
\def\@settitle{\begin{center}%
  \baselineskip14\p@\relax
    \bfseries
    \normalfont\LARGE
  \@title
  \end{center}%
}
\makeatother

\title[Correpondences for Hyperk\"ahler varieties]{Correspondences for hyperk\"ahler varieties with large Picard numbers}

\author{Ljudmila Kamenova, Abhinav Kumar}
\begin{abstract}
In this note, we explore the connection between hyperk\"ahler
manifolds with large Picard numbers and abelian varieties. In
particular, we are interested in Morrison's solution to the (modified)
Oda's conjecture: every K3 surface whose Picard group is large enough
(in a certain precise sense) must be related via an algebraic
correspondence to an abelian surface. We generalize this theorem to
the case of known examples of hyperk\"ahler manifolds such as pointed
Hilbert schemes on K3 surfaces.
\end{abstract}
\date{July 10, 2026}
\address{Ljudmila Kamenova, 
Department of Mathematics, Office 3-115, 
Stony Brook University, 
Stony Brook, NY 11794-3651, USA}
\email{kamenova@math.stonybrook.edu}

\address{Abhinav Kumar,
Department of Mathematics,
Stony Brook University, 
Stony Brook, NY 11794-3651, USA}
\email{thenav@gmail.com}

\bibliographystyle{amsalpha}
\maketitle
\section{Introduction}

An important aspect of the study of $K3$ surfaces is the relationship
to abelian surfaces. Namely, given an abelian surface $A$, the
resolution of the quotient $A/\{\pm 1\}$ is a $K3$ surface, {\it a
  Kummer $K3$ surface}.  Kummer $K3$ surfaces play an essential role
in the study of all $K3$ surfaces, for instance they are dense in the
moduli space of $K3$ surfaces, and one can use this fact to prove
Torelli-type theorems for $K3$ surfaces. \\

A Kummer $K3$ surface has $\rank(NS(X))\ge 17$, or equivalently,
$\rank(T(X))\le 5$. A remarkable theorem of Morrison \cite{mor} shows
that, in fact, every $K3$ surface with a large Picard number (e.g.,
$\rho(X)\ge 19$) is closely related to a Kummer surface.  More
precisely, suppose that there is a primitive embedding $\phi: T(X)
\hookrightarrow U^3$ of lattices, where $U$ is the hyperbolic lattice
of rank $2$. Equivalently, suppose that there is an embedding
$E_8(-1)^2 \hookrightarrow NS(X)$. Then there exists a Nikulin
involution on $X$ with quotient birational to $\Kum(A)$, such that the
quotient maps from $X$ and $A$ induce a Hodge isometry $T(S)\cong
T(A)$. Such a structure on $X$ is labelled a Shioda-Inose
correspondence with the abelian surface $A$. \footnote{The abelian
surface is not necessarily unique. However, there are only finitely
many choices; see Theorem \ref{thm:fm-partners} in the appendix.} \\

In this note, we are interested in similar questions for
\hk manifolds, which are natural higher dimensional
generalizations of $K3$ surfaces. There are two guiding principles
here: (1) a \hk manifold $X$ is ``essentially" determined
by its transcendental lattice $T(X)$, or rather its Hodge structure
along with the embedding $T(X)\hookrightarrow H^2(X,Z)$; see
Proposition~\ref{prop:TX-torelli} below for a more precise
formulation, and (2) $T(X)$ can be ``small" enough (for example, when
it has rank $2$) to embed into the second cohomology of an abelian
variety, and thus it can be assumed $T(X)\cong T(A)$ for
some\footnote{Once again, not necessarily unique, but determined up to
a finite choice.} abelian variety $A$. By the Hodge conjecture, $X$
and $A$ are related by a correspondence inducing $T(X)_\Q\cong
T(A)_\Q$. The purpose of this paper is to explore this correspondence
more explicitly in the case of the known \hk manifolds with
large Picard number in the spirit of Morrison's results for
$K3$ surfaces. \\

All the known \hk manifolds ($K3^{[n]}$, $\Kum_n$, OG6 and
OG10) arise from moduli of vector bundles/sheaves on K3 surfaces and
abelian surfaces. Thus, restricting to certain loci in moduli, they
will be naturally in correspondence with symplectic surfaces: either
$K3$ or abelian surfaces.  Recently, there has been renewed interest
in \hk manifolds with large Picard number and their
relation to K3 and abelian surfaces -- see esp. \cite{po} and
\cite{pm} (also \cite{L-max} for a different related question). Here,
we sharpen those results a bit, and focus on understanding the
connection to abelian surfaces in a uniform way. \\



Returning to K3 surfaces, in the situation of Morrison's theorem,
there is a Shioda-Inose correpondence between a K3 surface $X$ such
that $T(X) \hookrightarrow U^3$, and an abelian surface $A$, through
the common quotient $\Kum(A)$. More generally, one can ask under which
conditions there is a correspondence between $X$ and an abelian
surface $A$. This is described by the following conjecture of Oda,
modified and proved by Morrison (see the addendum at the end of
\cite{mor}, using foundational results of Mukai on the moduli of
sheaves on K3 surfaces \cite{muk}).

\begin{thm} (\cite{mor, muk}) \label{om}
Let $X$ be an algebraic K3 surface, and suppose that there is an
embedding $\phi: T(X) \otimes \Q \hookrightarrow U^3 \otimes \Q$ of
$\Q$-lattices. Then there exists an abelian surface $A$ and a
correspondence between $X$ and $A$ which induces a Hodge isometry $T(X)
\otimes \Q \cong T(A) \otimes \Q$.
\end{thm}
Note that the necessary condition is obvious (since $T(A)
\hookrightarrow H^2(A,\Q) \cong U^3$), while the theorem is implied by
the Hodge conjecture. Morrison's (unconditional) proof of the theorem
is as follows: suppose $\phi$ is the embedding as in the hypothesis,
and let $T = U^3 \cap \phi(T(X) \otimes \Q)$. Then $T$ is a lattice
with a primitive embedding in $U^3$, and the right
signature. Therefore, there is a K3 surface $Y$ with $T(Y) \cong T$,
and such that $Y$ is related to an abelian surface $A$ by a
Shioda-Inose correspondence. Since $T(X) \otimes \Q \cong T(Y) \otimes
\Q$ is a Hodge-isometry, and $T(X)$ and $T(Y)$ have rank $\leq 5$, the
K3 surfaces have Picard number $\geq 17$. In particular, they satisfy
the hypothesis of Mukai's theorem (which only requires Picard number
$\geq 11$), implying $X$ and $Y$ are related by a correspondence, and
therefore so are $X$ and $A$.\\

We remark that Buskin \cite{bus} generalized Mukai's theorem (removing
the Picard number $\geq 11$ condition) to show that if two K3 surfaces
$X$ and $Y$ are such that there is a Hodge isometry $H^2(X,\Q) \cong
H^2(Y,\Q)$, then there exists an algebraic correspondence between
$X$ and $Y$ inducing the Hodge isometry. \\

\begin{rmk} 
We note in passing that the condition $T(X) \otimes \Q \hookrightarrow
U^3 \otimes \Q$ is automatically satisfied when $\rank(T(X)) \in
\{2,3\}$; in fact then $T(X) \hookrightarrow U^3$ (see \cite[Cor
  2.5]{mor}). However, when $\rank(T(X)) \in \{4,5\}$, there are
non-trivial arithmetic conditions on $T(X)$ for it to be rationally
embeddable in three copies of the hyperbolic plane. For instance, when
$T$ has rank $4$ (signature $(2,2)$), it must represent $0$ (i.e.,
have an isotropic vector) in order for $T \otimes \Q \hookrightarrow
U^3 \otimes \Q$; on the other hand, any even $T$ of signature $(2,2)$
is $T(X)$ for some K3 surface $X$. Similarly, when $T$ has rank $5$
(signature $(2,3)$), it can be shown that $T(X) \otimes \Q
\hookrightarrow U^3 \otimes \Q$ iff $T(X)$ has an isotropic sublattice
of rank $2$, i.e., two vectors $u, v$ such that $u^2 = v^2 = u \cdot v
= 0$.\footnote{For a proof of these criteria see
Lemma~\ref{lem:embedding-in-h2abelian} below.}

\end{rmk}

The question that we are addressing here is that of finding a
correspondence between known \hk manifolds $X$ of large Picard ranks
and abelian surfaces $A$.

Our generalization of Morrison's solution to Oda's conjecture to \hk
varieties of \kntype~ is the following.

\begin{thm}
Let $X$ be a projective \hk manifold of \kntype. If there is an
embedding $\phi: T(X) \otimes \Q \hookrightarrow U^3 \otimes \Q$ of
$\Q$-lattices, then there exists an abelian variety $A$ and a
correspondence between $X$ and $A$ which induces a Hodge isometry $T(X)
\otimes \Q \cong T(A) \otimes \Q$. \\
\end{thm}

\begin{rmk}
One can show, similar to the remark above, that the hypothesis of the
theorem above is automatically satisfied if $\rho \in \{20,21\}$,
whereas for $\rho = 19$, it is satisfied iff $T(X)$ has an isotropic
vector, and for $\rho = 18$, iff $T(X)$ has an isotropic 2-dimensional
sublattice.
\end{rmk}

\hfill 

{\bf Acknowledgements:} We are grateful to Radu Laza, Eyal Markman and
Giovanni Mongardi for interesting conversations related to this
project. LK is partially supported by a grant/award
SFI-MPS-TSM-00013537 from the Simons Foundation International.

\section{Preliminaries}

A {\it hyperk\"ahler manifold} is a simply connected compact K\"ahler
manifold $X$ with a holomorphic symplectic form $\sigma \in H^0(X,
\Omega_X^2)$ which is unique up to rescaling. The known classes of
hyperk\"ahler examples are: Hilbert schemes of $n$ points on a $K3$
surface and their deformations (that is, of \kntype), generalized
Kummer varieties $\Kum_n$ in dimension $2n$, and the exceptional
examples due to O'Grady OG6 and OG10 in dimensions $6$ and $10$.

\begin{defn}
Let $X$ be a hyperk\"ahler manifold. The transcendental lattice of $X$
is the lattice $T(X) = NS(X)^\perp$, where the orthogonal complement
is taken in terms of the Beauville-Bogomolov-Fujiki form $q_X$. As a
lattice, $T(X)$ is endowed with the Beauville-Bogomolov-Fujiki form
$q_X$ restricted to $T(X)$ and the weight-two Hodge structure induced
by the lattice $(H^2(X, \mathbb Z), q_X)$.
\end{defn}

\begin{defn}
A pure integral Hodge structure $V$ of weight two is of K3-type if
$\dim_{\mathbb C} (V_{\mathbb C}^{2,0}) = 1$ and $V_{\mathbb C}^{p,q}
=0$ if $|p - q| > 2$.  Let $X$ be a projective hyperk\"ahler manifold
and let $(T,q)$ be a Hodge structure of K3-type.  We say that $X$ is
induced by $T$ if there exists a Hodge isometry $(T(X), q_X) \cong
(T,q)$, where $q_X$ is the Beauville-Bogomolov-Fujiki form of $X$.
\end{defn}

We recall the following proposition from \cite[Prop 2.12, Cor
  2.13]{po}, which makes precise the notion that the transcendental
lattice $T(X)$ with its Hodge structure controls the isomorphism class
of \hk $X$, up to finitely many choices:
\begin{prop} \label{prop:TX-torelli}
Let $(T,q)$ be a fixed Hodge structure of K3 type, and fix a lattice
$\Lambda$ of rank $\geq 5$. Then the set
\[
\{X \textrm{ \hk with } H^2(X,\Z) \cong \Lambda \textrm{ and induced by } (T,q) \}/\cong_{\textrm{iso}}
\]
is finite.
\end{prop}

In particular, if $(T,q)$ is the transcendental lattice $T(X_0)$ of a
fixed HK (say of K3$^{[n]}$-type for a fixed $n$), along with its
Hodge structure, the proposition says that the set of $X$ of type
K3$^{[n]}$-type with $T(X)$ Hodge-isometric to $T(X_0)$ is finite, up
to isomorphism.


Any hyperk\"ahler manifold of \kntype ~ with Picard rank $\rho(X) \geq
4$ is isomorphic to a moduli space of twisted stable sheaves on a $K3$
surface (and this is a sharp bound on the Picard rank) due to the
results of Prieto-Montanez \cite{pm}. If the Picard rank $\rho(X) \geq
13$, such a $K3$ surface is unique. More precisely, we have the
following results due to Markman \cite{mar}, see also Piroddi-Ortiz
\cite[Theorem 3.7, Corollary 3.9]{po}.

\begin{thm}
Let $X$ be a projective hyperk\"ahler manifold of \kntype ~ and $S$ a
projective $K3$ surface. Then $X$ is induced by the transcendental
lattice $T(S)$ if and only if $X$ is birational to the Mukai moduli
space $M_v(S,H)$ for some Mukai vector $v$ and a $v$-generic
polarization $H$.
\end{thm}

\begin{cor}
Every hyperk\"ahler manifold of \kntype ~ with Picard rank $\rho(X)
\geq 13$ is induced by a unique $K3$ surface $S$.
\end{cor}

In analogy to the results above, in \cite[Theorem 3.15, Corollary
  3.16]{po}, Piroddi-Ortiz also explore the generalized Kummer case.

\begin{thm}
Let $X$ be a hyperk\"ahler manifold of $\Kum^n$-type. Then $X$
is induced by $T(A)$ for an abelian surface $A$ if and only if $X$ is
birational to the generalized Kummer variety $K_v(A,H) =
\mathrm{Alb}^{-1} (0)$ associated with $A$, or to $K_v(A^\vee,H)$
associated with $A^\vee$, for some $v$-generic polarization $H$.
\end{thm}

\begin{cor}
Every hyperk\"ahler manifold of $\Kum^n$-type with Picard rank
$\rho(X) \geq 4$ is induced by a unique abelian surface $A$ or its
dual $A^\vee$.
\end{cor}

\section{Main result}

The generalization of Morrison's solution to Oda's conjecture to \hk
varieties of \kntype~ is as follows:
\begin{thm} \label{thm:oda-morrison-HK}
Let $X$ be a projective \hk manifold of \kntype. If there is an
embedding $\phi: T(X) \otimes \Q \hookrightarrow U^3 \otimes \Q$ of
$\Q$-lattices, then there exists an abelian variety $A$ and a
correspondence between $X$ and $A$ which induces a Hodge isometry $T(X)
\otimes \Q \cong T(A) \otimes \Q$. \\
\end{thm}

\begin{proof}
  Since $\rank(T(X)) = \dim_\Q T(X)\otimes \Q \leq 6$, we have that
  $T(X)$ has signature $(2,k)$, for $0 \leq k \leq 4$. By Nikulin's
  theorem~\ref{uniqueembedding}, we have a (unique) primitive
  embedding $i: T(X) \hookrightarrow L = \Lambda_{K3} = E_8(-1)^2
  \oplus U^3$. According to Morrison's Corollary 1.9, there is a K3
  surface $S$ with $T(S) \cong T(X)$. In order to show that we can make
  choices such that $T(S) \cong T(X)$ is in fact a Hodge isometry, we
  adapt Morrison's proof (and notation) of Corollary 1.9, as follows.

  The lattice $T(X)$ comes equipped with a Hodge structure $T(X) \otimes
  C \cong T^{2,0} \oplus T^{1,1} \oplus T^{0,2}$, where $T^{2,0} = \C
  \omega$ is $1$-dimensional. Let $S = i(T(X))^\perp$ in $L$. Define
  $L^{2,0} = \C i(\omega)$ and $L^{0,2} = \C i(\bar{\omega})$, and
  $L^{1,1}$ the orthogonal complement of $L^{2,0} \oplus L^{0,2}$ in
  $L \otimes \C$. Let $\Sigma = L^{1,1} \cap (L \otimes \R)$, and $N =
  \Sigma \cap L$. It is easy to see that $N$ and $i(T(X))$ are orthogonal
  complements in $L$, and the quadratic form $q$ on $L$, extended to
  $L \otimes \R$, has signature $(1,19)$ on the subspace $\Sigma$.
  Also, $L \otimes \C = L^{2,0} \oplus L^{1,1} \oplus L^{0,2}$, is a
  Hodge decomposition. Choosing a component of the light cone $\{ x
  \in L^{1,1} \cap (L \otimes R) : q(x) > 0\}$ makes it a signed Hodge
  structure. By the surjectivity of the period map (see for example,
  \cite[Theorem 1.7]{mor}), there is a K3 surface $S$ and a signed
  Hodge isometry $\phi : H^2(S,\Z) \rightarrow L$. It follows that
  $\phi|_{NS(S)}$ gives an isometry of $NS(S)$ with $L^{1,1} \cap L =
  \Sigma \cap L = N$, which has signature $(1,\rho-1)$ for some $\rho
  \leq 20$, and therefore $S$ is algebraic. Finally, $\phi|_{T(S)}$
  gives a Hodge isometry $T(S) \cong i(T(X))$.
  
  By \cite[Theorem 3.7]{po} (essentially proved by Markman in
  \cite{mar}), since $X$ is induced by $T(S)$, we have that $X$ is
  birational to a moduli space $M_v(S,H)$ for some
  primitive\footnote{The theorem as stated does not mention that $v$
  is primitive; however, the proof constructs such a primitive vector
  $v$.} vector $v \in \tilde{H}(S)$ and a $v$-generic polarization
  $H$. (Their remark 3.8 further shows that $X$ is a moduli space of
  semistable sheaves for some Bridgeland stability condition.)

 Next, by \cite[Theorem 3.5]{po} (due to Mukai, O'Grady, and
 Yoshioka), the moduli space $M_v(S,H)$ is deformation equivalent to
 $S^{[n]}$, where $2n = \langle v,v \rangle + 2$.  By \cite[Theorem
   1.1]{Markman_rational}, it is enough to construct a correspondence
 between $Y = S^{[n]}$ and $S$. Now we have the maps $Y
 \xrightarrow{HC} \Sym^{n}(S)$, and $Z = S^n \xrightarrow{\pi}
 \Sym^n(S)$. On $Z \times S = S^{n+1}$ we have the obvious
 correspondence $\Psi = \sum_i \pi_{i,n+1}^*(\Delta)$ where $\Delta
 \subset S \times S$ is the diagonal. It induces a diagonal map on
 $H^2$. So we can push it forward by $\pi$ and pull it back by $HC$,
 the Hilbert-Chow morphism, to get the correspondence from $Y$ to $S$.



  Note that by Oda-Morrison's theorem for K3 surfaces (\ref{om}),
  since $T(S) \otimes \Q \cong T(X) \otimes \Q \hookrightarrow U^3
  \otimes \Q$, we already know there is a correspondence between $S$
  and $A$ inducing a Hodge isometry $T(S) \otimes \Q \cong T(A) \otimes
  \Q$, and composing with the previous correspondence $X \sim S$, we
  prove Oda-Morrison's conjecture for the \hk manifold
  $S^{[n]}$.  By Markman's \cite[Theorem 1.1]{Markman_rational}, since
  we are in the projective case, there exists an algebraic
  correspondence between $X$ of \kntype~ and $Y=S^{[n]}$, and
  therefore between $X$ and $A$, inducing a Hodge isometry $T(X)
  \otimes \Q \cong T(A) \otimes \Q$.

\end{proof}

\begin{lem} \label{lem:embedding-in-h2abelian}
Let $t$ be a quadratic space over $\Q$ of signature $(2,k)$, and let
$u$ be the hyperbolic plane over $\Q$. Then
\begin{enumerate}
\item If $\dim t \leq 3$, (i.e., if $k \leq 1)$, then $t
  \hookrightarrow u^3$.
\item If $k = 2$, then $t \hookrightarrow u^3$ iff $t$ has an
  isotropic vector.
\item If $k = 3$, then $t \hookrightarrow u^3$ iff $t$ has a
  2-dimensional isotropic subspace.
\end{enumerate}
\end{lem}

\begin{rmk}
In the lemma, $k=4$ is impossible due to signature reasons. Also, the
proof of the lemma is analogous to that of \cite[Cor 2.6]{mor}
\end{rmk}
\begin{proof}
The first part follows from the fact that $u$ represents every
rational number. In the case when $\dim t = 3$, we can diagonalize $t
= \langle a,b,c \rangle$, we can find vectors $\alpha, \beta, \gamma$ in each
of the three orthogonal copies of $u$ with $\langle \alpha, \alpha
\rangle = a$ etc. The case of $\dim t = 2$ is even simpler; one can
embed $t$ in $u^2$, in fact.

For the second part, first note that if $t$ has an isotropic vector
$\alpha$, then we can find $\beta \in t$ such that $\langle \alpha,
\beta \rangle = 1$; then $\Q \alpha + \Q \beta$ is a hyperbolic
plane. Therefore, $t = u \oplus t'$ for some $t'$ of signature
$(1,1)$, which embeds into $u^2$ by the proof of part (1). Therefore
$t \hookrightarrow u^3$. Conversely, if $t \hookrightarrow u^3$, then
$t^\perp$ has signature $(1,1)$ and embeds in $u^2$ by part (1). Then
$u^3 = t \oplus t^\perp \hookrightarrow t \oplus u^2$, and so $u
\hookrightarrow t$ by Witt cancellation.

Finally, if $t$ has dimension $5$, and $t$ has an isotropic subspace
of dimension $2$ generated by $\alpha, \alpha'$, then as in the
previous paragraph, we can find $\beta, \beta'$ such that these four
vectors span $u^2$. Therefore, $t = u^2 \oplus t'$ for some
1-dimensional $t'$, which embeds in $u$. Therefore $t \hookrightarrow
u^3$. Conversely, if $t \hookrightarrow u^3$, then $t^\perp$ is
1-dimensional and embeds in $u$. Therefore, $u^3 = t \oplus t'
\hookrightarrow t \oplus u$, and then $u^2 \hookrightarrow t$ by Witt
cancellation.

\end{proof}

We obtain the following corollary by combining
Theorem~\ref{thm:oda-morrison-HK} and
Lemma~\ref{lem:embedding-in-h2abelian}.
\begin{cor}
Let $X$ by a projective \hk manifold of \kntype. Then if any of the
following conditions are met, there is an abelian variety $A$ and a
correspondence between $X$ and $A$ inducing a Hodge isometry $T(X)
\otimes \Q \cong T(A) \otimes \Q$.
\begin{itemize}
    \item $\rho(X) \in \{ 20, 21 \}$.
    \item $\rho(X) = 19$ and $T(X)$ has an isotropic vector.
    \item $\rho(X) = 18$ and $T(X)$ has a 2-dimensional isotropic
      sublattice (i.e., $\Z x + \Z y$ such that $x^2 = y^2 = x\cdot y
      =0$).
\end{itemize}
\end{cor}

\hfill

\begin{rmk}
To summarize, now we know how to handle the $K3^{[n]}$ type case: for
large Picard numbers, $X$ is precisely the Mukai moduli space $M_v(S)$
for a unique $K3$ surface $S$ with a large Picard number. By
Morrison's results, there is a correspondence between $S$ and an
abelian surface, which in turn corresponds to an abelian variety.  The
other known cases of of hyperk\"ahler manifolds can be treated
similarly.  The generalized Kummer varieties $Kum_n$ already arise
from a Kummer-type construction due to Beauville's classical
construction \cite{Beauville}.  The exceptional O'Grady examples OG10
and OG6 are related (up to an index 2) to a $K3^{[5]}$-type and a
$K3^{[3]}$-type hyperk\"ahler manifold, respectively. The example OG6
is a quotient of $K3^{[3]}$ by a birational involution
(\cite{MRS}). For the exceptional example OG10, if the transcendental
lattice $T(X)$ contains a copy of the hyperbolic lattice $U$, then $X$
is of LSV-type (\cite{LSV}), and then it is in correspondence with a
cubic fourfold $Z$, and in turn, for large $NS(Z)$ one obtains a
correspondence to a $K3$ surface (as noticed in Laza's paper
\cite{L-max}).
\end{rmk}

\appendix
\section{Some lattice theory}

We briefly recall some background on lattices.

\begin{defn}
  A {\em lattice} will denote a finitely generated free abelian group
  $\Lambda$ equipped with a symmetric bilinear form $B: \Lambda \times
  \Lambda \rightarrow \Z$.
\end{defn}
We abbreviate the data $(\Lambda, B)$ to $\Lambda$ sometimes, when the
form is understood, and we interchangeably write $u\cdot v = \langle
u, v \rangle = B(u,v)$ and $u^2 = \langle u,u \rangle = B(u,u)$ for $u
\in \Lambda$.

The {\em signature} of the lattice is the real signature of the form
$B$, written $(r_+, r_-, r_0)$ where $r_+$, $r_-$ and $r_0$ are the
number of positive, negative and zero eigenvalues of $B$, counted with
multiplicity. We say that the lattice is {\em non-degenerate} if the
form $B$ has zero kernel, i.e. $r_0 = 0$. In that case, the signature
is abbreviated to $(r_+, r_-)$. We say $\Lambda$ is {\em even} if $x^2
\in 2\Z$ for all $x \in \Lambda$.

Let $\Lambda$ be a non-degenerate lattice. The {\em discriminant} of
the lattice is $\disc(\Lambda) := |\det(B)|$. The lattice is said to
be {\em unimodular} if its discriminant is $1$. From now on, we
restrict attention to non-degenerate lattices. We may think of
$\Lambda$ as embedded in the real or p-adic space $\Lambda \otimes \R$
or $\Lambda \otimes \Q_p$, and extend the bilinear form to the ambient
space.

\begin{defn}
  The {\em dual lattice} of $\Lambda$ is $\Lambda^* =
  \textrm{Hom}(\Lambda, \Z)$.  We can identify $\Lambda^*$ with a
  subgroup of $\Lambda_\Q = \Lambda \otimes \Q$, using the bilinear
  form $B$ to identify $\Lambda_\Q$ with its $\Q$-dual. It is easy to
  show that $\disc(\Lambda) = [ \Lambda^* : \Lambda ]$.
\end{defn}

\begin{defn}
  The finite abelian group $\Lambda^*/\Lambda$ is called the {\em
    discriminant group} $A_{\Lambda}$.
\end{defn}

Let $\Lambda$ be an even non-degenerate lattice.
\begin{defn} A stronger invariant of the
lattice is its {\em discriminant form} $q_{\Lambda}$, which is defined
as follows. The form $B$ on $\Lambda$ induces a $\Q$-valued bilinear
symmetric form on $\Lambda^*$, and consequently also induces a
quadratic form on $A_{\Lambda}$ as follows. The induced form
$\Lambda^* \times \Lambda^* \rightarrow \Q$ takes $\Lambda^* \times
\Lambda$ into $\Z$, and the diagonal of $\Lambda \times \Lambda$ to
$2\Z$. Therefore we get an induced symmetric form $b: A_{\Lambda}
\times A_{\Lambda} \rightarrow \Q/\Z$ and a quadratic form $q:
A_{\Lambda} \rightarrow \Q/2\Z$ such that for all $n \in \Z$ and all
$a,b \in \Lambda$, we have
$$
q(na) = n^2q(a)
$$
$$
q(a+a') -q(a)-q(a') \equiv 2b(a,a') \mod 2\Z
$$
This data $(A_{\Lambda}, b,q)$ will be abbreviated to $q_{\Lambda}$.
\end{defn}

Note that the discriminant of the lattice is just the size of the
discriminant group. We shall let $l(A)$ denote the minimum number of
generators of an abelian group $A$. Note that $l(A_{\Lambda}) \leq
\textrm{ rank}(\Lambda^*) = \textrm{ rank}(\Lambda)$. For a unimodular
lattice $\Lambda$, we have $l(A_{\Lambda}) = 0$.

The discriminant form of a unimodular lattice is trivial, and if $M
\subset L$ is a {\em primitive} embedding of non-degenerate even
lattices (that is, $L/M$ is a free abelian group), with $L$
unimodular, then we have
$$
q_{M^{\perp}} = -q_M
$$

For a lattice $\Lambda$ and a real number $\alpha$, we denote by
$\Lambda(\alpha)$ the lattice which has the same underlying group but
with the bilinear form scaled by $\alpha$. The lattice $\Z\langle
\alpha \rangle$ of rank one with a generator of norm $\alpha$ will be
abbreviated $\langle \alpha \rangle$.

A {\em root} of a positive definite lattice, we will mean an element
$x$ such that $x^2 = 2$, whereas for a negative-definite or indefinite
lattice, we will mean an element $x$ such that $x^2 = -2$.  A {\em
  root lattice} is a lattice that is spanned by its roots.

The familiar root lattices (positive definite by convention) are
listed below:

\begin{itemize}
\item $A_n$ : This is the $n$-dimensional simplex lattice $\{x \in
  \Z^{n+1} \big| \sum x_i = 0 \}$ inside the sum-zero hyperplane in
  $R^{n+1}$. It has $n(n+1)$ roots, and discriminant $n+1$.
  
\item $D_n$ : The $n$-dimensional checkerboard lattice is $\{x \in
  \Z^n \big| \sum\limits_{i=1}^n x_i \equiv 0 \mod 2\}$. It has
  $2n(n-1)$ roots, and discriminant $4$.

\item $E_8$ : This is an "exceptional root lattice"; one realization
  of $E_8$ is as the span of $D_8$ and the all-halves vector
  $(1/2,\ldots,1/2)$. It has $240$ roots, and is unimodular.

\item $E_7$ : taking the orthogonal complement of any root in $E_8$
  gives us the root lattice $E_7$. It has $126$ roots, and
  discriminant $2$.

\item $E_6$ : taking the orthogonal complement of any $A_2$ sublattice
  of $E_8$ (i.e., spanned by two roots $e_1,e_2$ such that $e_1 \cdot
  e_2 = -1$), gives us $E_6$. It has $72$ roots, and discriminant $3$.

\end{itemize}

The ADE classification theorem of root lattices states that any
positive definite root lattice can be decomposed into an orthogonal
direct sum of the root lattices $A_n, D_n, E_{\{6,7,8\}}$ listed
above.

Let $U$ be the {\em hyperbolic plane}, i.e. the indefinite rank $2$
lattice whose matrix is
\[ \left(\begin{array}{cc}
0 & 1 \\
1 & 0 
\end{array}\right) \]
Note that $U(-1) \cong U$.

Positive definite even unimodular lattices only exist in dimensions
which are multiples of $n$, and their number grows at least doubly
exponentially in the dimension. (More generally, the
Smith-Minkowski-Siegel mass formula can be used to enumerate or
estimate the number of lattices of positive definite lattices of a
given dimension with specified local character at all primes, i.e., in
a genus.)

In contrast, the structurex of indefinite unimodular lattices is quite
simple to describe.

\begin{thm} \label{Milnor} {\em (Milnor \cite{mil})}
Let $\Lambda$ be an indefinite unimodular lattice. If $\Lambda$ is
even, then $\Lambda \cong E_8(\pm 1)^m \oplus U^n$ for some $m$ and
$n$. If $\Lambda$ is odd, then $\Lambda \cong \langle 1 \rangle ^m
\oplus \langle -1 \rangle^n$ for some $m$ and $n$.
\end{thm}

\begin{thm} \label{uniquelattice} {\em (Kneser \cite{kn}, Nikulin \cite{nik})}
Let $L$ be an even lattice with signature $(s_+, s_-)$ and
discriminant form $q_L$ such that
\begin{enumerate}
\item $s_+ > 0$
\item $s_- > 0$
\item $l(A_L) \leq \textrm{rank}(L) -2$.
\end{enumerate}
Then $L$ is the unique lattice with that signature and discriminant
form, up to isometry.
\end{thm}

\begin{prop}
Let $M$ be a unimodular lattice. Then if $M \subset L$, we have an
orthogonal decomposition $L = M \oplus M^{\perp}$.
\end{prop}


\begin{thm} \label{uniqueembedding} {\em (Nikulin \cite{nik} Theorem 1.14.4)}
Let $M$ be an even lattice with invariants $(t_+, t_-, q_M)$ and let
$L$ be an even unimodular lattice of signature $(s_+, s_-)$. Suppose
that
\begin{enumerate}
\item $t_+ < s_+$.
\item $t_- < s_-$.
\item $l(A_M) \leq \textrm{rank}(L) - \textrm{rank}(M) - 2$. 
\end{enumerate}
Then there exists a unique primitive embedding of $M$ into $L$, up to
automorphisms of $L$.
\end{thm}





\begin{rmk}
Note that $l(A_M) \leq rank(M)$, so the third condition in
Theorem~\ref{uniqueembedding} may be replaced by the sufficient
condition $\rank(M) \leq rank(L)/2 - 1$.
\end{rmk}

\begin{thm}  \label{thm:fm-partners}
Fix an abelian surface $A$. There are only finitely many abelian
surfaces $B$ up to isomorphism such that there is a Hodge isometry
$T(A) \cong T(B)$.
\end{thm}
It is a theorem of Mukai and Orlov that $A$ and $B$ satisfy this
property iff they are Fourier-Mukai partners, i.e., their derived
categories of bounded complexes of coherent sheaves are
equivalent. \cite{Galluzzi-Lombardo}.

\begin{proof}
Denote the abstract lattice $T(A) \cong T(B)$ by $M$, and let $L$ be the
class of the abstract lattice $H^2(A,\Z) \cong H^2(B,Z) \cong
U^3$. Then by \cite{Miranda_Morrison}, there are finitely many classes
of embeddings $M \hookrightarrow L$, up to isometries of $L$. If the
embeddings $T(A) \hookrightarrow H^2(A,\Z)$ and $T(B) \hookrightarrow
H^2(B,\Z)$, identified by $\phi_T: T(A) \cong T(B)$, are the same up to
isometries of $L$, then it follows that the Hodge isometry $\phi_T$
can be extended to a Hodge isometry $\phi : H^2(A,\Z) \rightarrow
H^2(B,\Z)$ (since $NS(A) = T(A)^\perp$ lies in the $(1,1)$ part of the
Hodge decomposition and similarly for $B$) and then $A \cong B$ by
Torelli. Therefore, there are only finitely many possibilities for
$B$, up to isomorphism.
\end{proof} 

As a corollary, for a given K3 surface $X$ there are only finitely
many abelian surfaces $A$ related to $X$ by a Shioda-Inose structure
(since any other such $B$ would have $T(A) \cong T(X) \cong T(B)$ by
Hodge isometries). \\

We end with a brief mention of the theory of quadratic spaces over
$\Q$: it is significantly simpler than that of integral lattices. A
good reference is the book "Rational Quadratic Forms" \cite{Cassels}
by Cassels. We have used the following key lemma in the proof of
Lemma~\ref{lem:embedding-in-h2abelian}.

\begin{lem}[Witt cancellation]
Let $v, w_1, w_2$ be quadratic spaces over $\Q$, such that $v \oplus
w_1 \cong v \oplus w_2$. Then $w_1 \cong w_2$.
\end{lem}
 
\bibliography{KKL}

@article{nik,
	author = {Nikulin, V. V.},
	fjournal = {Izvestiya Akademii Nauk SSSR. Seriya Matematicheskaya},
	issn = {0373-2436},
	journal = {Izv. Akad. Nauk SSSR Ser. Mat.},
	mrclass = {10C05 (14G30 14J17 14J25 57M99 57R45 58C27)},
	mrreviewer = {I.\ Dolgachev},
	number = {1},
	pages = {111--177, 238},
	title = {Integer symmetric bilinear forms and some of their geometric applications},
	volume = {43},
	year = {1979}}

@article{kn,
	author = {Kneser, M.},
	doi = {10.1007/BF01900681},
	fjournal = {Archiv der Mathematik},
	issn = {0003-889X,1420-8938},
	journal = {Arch. Math. (Basel)},
	mrclass = {10.0X},
	mrreviewer = {B.\ W.\ Jones},
	pages = {323--332},
	title = {Klassenzahlen indefiniter quadratischer {F}ormen in drei oder mehr {V}er\"anderlichen},
	url = {https://doi.org/10.1007/BF01900681},
	volume = {7},
	year = {1956}}

@incollection{mil,
	author = {Milnor, John},
	booktitle = {Symposium internacional de topolog\'ia algebraica {I}nternational symposium on algebraic topology},
	mrclass = {55.00},
	mrreviewer = {H.\ Samelson},
	pages = {122--128},
	publisher = {Universidad Nacional Aut\'onoma de M\'exico and UNESCO, M\'exico},
	title = {On simply connected {$4$}-manifolds},
	year = {1958}}

@article{mar,
	author = {Markman, E.},
	doi = {10.1142/S0129167X10005957},
	fjournal = {International Journal of Mathematics},
	issn = {0129-167X,1793-6519},
	journal = {Internat. J. Math.},
	mrclass = {14C05 (14C34 14J28 14J60 32G05 53C26)},
	mrreviewer = {Zhenbo\ Qin},
	number = {2},
	pages = {169--223},
	title = {Integral constraints on the monodromy group of the hyper{K}\"ahler resolution of a symmetric product of a {$K3$} surface},
	url = {https://doi.org/10.1142/S0129167X10005957},
	volume = {21},
	year = {2010}}

@article{Galluzzi-Lombardo,
 author = {Galluzzi, Federica and Lombardo, Giuseppe},
 title = {Correspondences between {{\(K3\)}} surfaces (with an appendix by {Igor} {Dolgachev})},
 fjournal = {Michigan Mathematical Journal},
 journal = {Mich. Math. J.},
 issn = {0026-2285},
 volume = {52},
 number = {2},
 pages = {267--277},
 year = {2004},
 language = {English},
 doi = {10.1307/mmj/1091112075},
 zbMATH = {2112901},
 Zbl = {1067.14032}
}

@article{Miranda_Morrison,
 author = {Miranda, Rick and Morrison, David R.},
 title = {The number of embeddings of integral quadratic forms. {II}},
 fjournal = {Proceedings of the Japan Academy. Series A},
 journal = {Proc. Japan Acad., Ser. A},
 issn = {0386-2194},
 volume = {62},
 pages = {29--32},
 year = {1986},
 language = {English},
 doi = {10.3792/pjaa.62.29},
 zbMATH = {3955007},
 Zbl = {0594.10012}
}

@article{bus,
	author = {Buskin, N.},
	doi = {10.1515/crelle-2017-0027},
	fjournal = {Journal f\"ur die Reine und Angewandte Mathematik. [Crelle's Journal]},
	issn = {0075-4102,1435-5345},
	journal = {J. Reine Angew. Math.},
	mrclass = {14J28 (14C30 32J15)},
	mrreviewer = {I.\ Dolgachev},
	pages = {127--150},
	title = {Every rational {H}odge isometry between two {$K3$} surfaces is algebraic},
	url = {https://doi.org/10.1515/crelle-2017-0027},
	volume = {755},
	year = {2019}}

@incollection{muk,
	author = {Mukai, S.},
	booktitle = {Vector bundles on algebraic varieties ({B}ombay, 1984)},
	isbn = {0-19-562014-3},
	mrclass = {14J28 (14D22 14F05)},
	mrreviewer = {Mei\ Chu\ Chang},
	pages = {341--413},
	publisher = {Tata Inst. Fund. Res., Bombay},
	series = {Tata Inst. Fund. Res. Stud. Math.},
	title = {On the moduli space of bundles on {$K3$} surfaces. {I}},
	volume = {11},
	year = {1987}}

@unpublished{pm,
	author = {Prieto-Montanez, Y.},
	note = {arXiv:2408.16610},
	title = {On Hyperk\"ahler manifolds of {$K3^{[n]}$}-type with large Picard number},
	year = {2024}}

@article{po,
	author = {Piroddi, B. and R\'ios Ortiz, \'A.},
	doi = {10.1016/j.jpaa.2024.107805},
	fjournal = {Journal of Pure and Applied Algebra},
	issn = {0022-4049,1873-1376},
	journal = {J. Pure Appl. Algebra},
	mrclass = {14J42},
	mrreviewer = {Giacomo\ Mezzedimi},
	number = {1},
	pages = {Paper No. 107805, 22},
	title = {On the transcendental lattices of hyper-{K}\"ahler manifolds},
	url = {https://doi.org/10.1016/j.jpaa.2024.107805},
	volume = {229},
	year = {2025}}

@article{mor,
	author = {Morrison, D. R.},
	doi = {10.1007/BF01403093},
	fjournal = {Inventiones Mathematicae},
	issn = {0020-9910,1432-1297},
	journal = {Invent. Math.},
	mrclass = {14J28 (14C30 14J05 14K10)},
	mrreviewer = {I.\ Dolgachev},
	number = {1},
	pages = {105--121},
	title = {On {$K3$}\ surfaces with large {P}icard number},
	url = {https://doi.org/10.1007/BF01403093},
	volume = {75},
	year = {1984}}

@article{L-max,
	author = {Laza, Radu},
	doi = {10.1090/proc/15430},
	fjournal = {Proceedings of the American Mathematical Society},
	issn = {0002-9939,1088-6826},
	journal = {Proc. Amer. Math. Soc.},
	mrclass = {14J42 (14E08 14J28 14J35)},
	mrnumber = {4273129},
	mrreviewer = {I.\ Dolgachev},
	number = {8},
	pages = {3209--3220},
	title = {Maximally algebraic potentially irrational cubic fourfolds},
	url = {https://doi.org/10.1090/proc/15430},
	volume = {149},
	year = {2021}}

@book{Cassels,
 author = {Cassels, J. W. S.},
 title = {Rational quadratic forms},
 fseries = {London Mathematical Society Monographs},
 series = {Lond. Math. Soc. Monogr.},
 volume = {13},
 year = {1978},
 publisher = {Academic Press, London},
 language = {English},
 zbMATH = {3613145},
 Zbl = {0395.10029}
}

@article{Markman_rational,
 author = {Markman, Eyal},
 title = {Rational {Hodge} isometries of hyper-{K{\"a}hler} varieties of {{\(K3^{[n]}\)}} type are algebraic},
 fjournal = {Compositio Mathematica},
 journal = {Compos. Math.},
 issn = {0010-437X},
 volume = {160},
 number = {6},
 pages = {1261--1303},
 year = {2024},
 language = {English},
 doi = {10.1112/S0010437X24007048},
 zbMATH = {7852901},
 Zbl = {1551.14161}
}

@article{Beauville,
 author = {Beauville, Arnaud},
 title = {Vari{\'e}t{\'e}s k{\"a}hleriennes dont la premi{\`e}re classe de {Chern} est nulle},
 fjournal = {Journal of Differential Geometry},
 journal = {J. Differ. Geom.},
 issn = {0022-040X},
 volume = {18},
 pages = {755--782},
 year = {1983},
 language = {French},
 doi = {10.4310/jdg/1214438181},
 zbMATH = {3853948},
 Zbl = {0537.53056}
}

@article{LSV,
 author = {Laza, Radu and Sacc{\`a}, Giulia and Voisin, Claire},
 title = {A hyper-{K{\"a}hler} compactification of the intermediate {Jacobian} fibration associated with a cubic 4-fold},
 fjournal = {Acta Mathematica},
 journal = {Acta Math.},
 issn = {0001-5962},
 volume = {218},
 number = {1},
 pages = {55--135},
 year = {2017},
 language = {English},
 doi = {10.4310/ACTA.2017.v218.n1.a2},
 zbMATH = {6826204},
 Zbl = {1409.14053}
}

@article{MRS,
 author = {Mongardi, Giovanni and Rapagnetta, Antonio and Sacc{\`a}, Giulia},
 title = {The {Hodge} diamond of {O}'{Grady}'s six-dimensional example},
 fjournal = {Compositio Mathematica},
 journal = {Compos. Math.},
 issn = {0010-437X},
 volume = {154},
 number = {5},
 pages = {984--1013},
 year = {2018},
 language = {English},
 doi = {10.1112/S0010437X1700803X},
 zbMATH = {6908312},
 Zbl = {1420.14095}
}
\end{document}